\documentclass[graybox]{svmult}

\usepackage{mathptmx}       
\usepackage{helvet}         
\usepackage{courier}        
\usepackage{type1cm}        
\usepackage{makeidx}         
\usepackage{graphicx}        
\usepackage{multicol}        

\usepackage{amssymb,amsfonts,amscd,amsmath,amscd}
\usepackage{hyperref}
\usepackage{url}
\usepackage{breakurl}

\input amssym.def

\input amssym.tex

\DeclareFontFamily{OT1}{pzc}{}
\DeclareFontShape{OT1}{pzc}{m}{it}%
             {<-> s * [0.900] pzcmi7t}{}
\DeclareMathAlphabet{\mathscr}{OT1}{pzc}%
                                 {m}{it}

\newcommand{\bburl}[1]{\textcolor{blue}{\url{#1}}}

\newcommand{\ol}[1]{\overline{#1}}

\numberwithin{equation}{section}

\newtheorem{thm}{Theorem}[section]

\newtheorem{que}[thm]{Question}

\renewcommand\thesection{\arabic{section}}

\newcommand\be{\begin{equation}}

\newcommand\ee{\end{equation}}

\newcommand\bea{\begin{eqnarray}}

\newcommand\eea{\end{eqnarray}}

\newcommand\bi{\begin{itemize}}

\newcommand\ei{\end{itemize}}

\newcommand\ben{\begin{enumerate}}

\newcommand\een{\end{enumerate}}

\newcommand\bc{\begin{center}}

\newcommand\ec{\end{center}}

\newcommand\ba{\begin{array}}

\newcommand\ea{\end{array}}

\newcommand{\N}{\mathbb{N}}

\newcommand\frakfamily{\usefont{U}{yfrak}{m}{n}}

\DeclareTextFontCommand{\textfrak}{\frakfamily}

\numberwithin{subsubsection}{subsection}

\begin{document}

%

\title*{Ramsey Theory Problems over the Integers: Avoiding Generalized Progressions}


\author{Andrew Best, Karen Huan, Nathan McNew, Steven J. Miller, Jasmine Powell, Kimsy Tor, and Madeleine Weinstein}
\institute{Department of Mathematics and Statistics, Williams College, andrewbest312@gmail.com; Department of Mathematics and Statistics, Williams College, klh1@williams.edu; Department of Mathematics, Dartmouth College, nathan.g.mcnew.gr@dartmouth.edu; Department of Mathematics and Statistics, Williams College, sjm1@williams.edu (Steven.Miller.MC.96@aya.yale.edu); Department of Mathematics, Northwestern University, jasminepowell2015@u.northwestern.edu; Department of Mathematics, Manhattan College, ktor.student@manhattan.edu; Department of Mathematics, Harvey Mudd College, mweinstein@g.hmc.edu}

%
%
%

%
%
\maketitle

\abstract{Two well studied Ramsey-theoretic problems consider subsets of the natural numbers which either contain no three elements in arithmetic progression, or in geometric progression. We study generalizations of this problem, by varying the kinds of progressions to be avoided and the metrics used to evaluate the density of the resulting subsets.\\ \
One can view a 3-term arithmetic progression as a sequence $x, f_n(x), f_n(f_n(x))$, where $f_n(x) = x + n$, $n$ a nonzero integer. Thus avoiding three-term arithmetic progressions is equivalent to containing no three elements of the form $x, f_n(x), f_n(f_n(x))$ with $f_n \in\mathcal{F}_{\rm t}$, the set of integer translations. One can similarly construct related progressions using different families of functions.  We investigate several such families, including geometric progressions ($f_n(x) = nx$ with $n > 1$ a natural number) and exponential progressions ($f_n(x) = x^n$).\\ \
Progression-free sets are often constructed ``greedily,'' including every number so long as it is not in progression with any of the previous elements. Rankin characterized the greedy geometric-progression-free set in terms of the greedy arithmetic set.  We characterize the greedy exponential set and prove that it has asymptotic density 1, and then discuss how the optimality of the greedy set depends on the family of functions used to define progressions.\\ \
Traditionally, the size of a progression-free set is measured using the (upper) asymptotic density, however we consider several different notions of density, including the uniform and exponential densities.}





\section{Background}\label{sec:bac}
A classic Ramsey-theoretic problem is to consider how large a set of integers can be without containing  3 terms in the set that are in arithmetic progression. In other words, no 3 integers in the set are of the form $a, a+n, a+2n$. An analogous problem involves looking at sets avoiding 3-term geometric progressions of the form $a, na, n^2a$. This question was first introduced by Rankin in 1961. More recently, Nathanson and O'Bryant \cite{Nato'Bry1} and the third named author \cite{McNew} have made further progress toward characterizing such sets and finding bounds on their maximal densities.


Progression-free sets are often constructed ``greedily'': Starting with an initial included integer, every successive number is included so long as doing so does not create a  progression involving any of the previously included elements. We consider two possible generalizations of the greedy arithmetic and geometric progresssion-free sets. Let $A_3^*=\{0,1,3,4,9,10,\dots\}$ be the greedy set of nonnegative integers free of arithmetic progressions. Note that $A_3^*$ consists of exactly those nonnegative integers with no digit 2 in their ternary expansions.  Let $G_3^*=\{1, 2, 3, 5, 6, 7, 8, 10, 11, 13, 14, 15, 16, 17, 19, 21, 22, 23, \dots\}$ be the greedy set of positive integers free of geometric progressions. In 1961, Rankin \cite{Rankin} characterized this set as the set of those integers where all of the exponents in their prime factorization are contained in $A_3^*$. We will use this characterization below to compute the size of Rankin's set with respect to various densities.

One can view arithmetic and geometric progressions as part of a larger class of functional progressions consisting of three terms of the form $x,f_{n}(x),f_{n}(f_{n}(x))$. From this perspective, a natural generalization of arithmetic and geometric progressions would be to let $f_n(x)=x^n$ and so consider exponential-progression-free sets. We characterize the greedy set in this case, which we call $E_3^*$. We show that its uniform density is $1$ (Theorem \ref{ud}) and the exponential density of the set of integers excluded from the greedy set $E_3^*$ is $1/4$ (Proposition \ref{ed}). 

Additionally, we consider the relationship between $G_3^*$ and $A_3^*$, namely that the geometric-progression-free set is constructed by taking those numbers with prime exponents in the arithmetic-progression-free set.  This leads us to consider iterating this idea so that in each step the permissible set of exponents comes from the prior iteration. We show that the asymptotic densities of the sets produced in each iteration of this construction approach 1 (Theorem \ref{uad}), but that each has a lower uniform density of 0 (Theorem \ref{lud}).

\section{Comparing Asymptotic and Uniform Densities}\label{sec:asym-uni}

\subsection{Definitions}\label{sec:defns}

The density most frequently encountered is the asymptotic density, $d(A)$. When the asymptotic density does not exist, the upper asymptotic density, $\overline{d}(A)$, and the lower asymptotic density, $\underline{d}(A)$ can be used instead. Their definitions are as follows.

\begin{definition}The asymptotic density of a set $A \subseteq \N$, if it exists, is defined to be
\be\label{eq:density}
d(A) \ = \ \lim_{N \to \infty} \frac{|A \cap \{1, \dots, N\}|}{N}.
\ee
The upper asymptotic density of a set $A \subseteq \N$ is defined to be
\be\label{eq:upperdensity}
\overline{d}(A) \ = \ \limsup_{N \to \infty} \frac{|A \cap \{1, \dots, N\}|}{N},
\ee and the lower asymptotic density of a set $A \subseteq \N$ is defined to be
\be\label{eq:lowerdensity}
\underline{d}(A) \ = \ \liminf_{N \to \infty} \frac{|A \cap \{1, \dots, N\}|}{N}.
\ee
\end{definition}

Using Rankin's characterization of $G_3^*$ in Section \ref{sec:bac},  its asymptotic density can be computed as follows:
\begin{align} \label{rankinset}
d(G_3^*) \ = \ \prod_p \left(\frac{p-1}{p}\right)\sum_{i\in A_3^*}\frac{1}{p^i}
\ & = \ \prod_{p}\left(1-\frac{1}{p}\right) \prod_{i=0}^{\infty}\left(1+\frac{1}{p^{3^i}}\right)\nonumber\\
\ & = \ \prod_{p}\left(1-\frac{1}{p^2}\right) \prod_{i=1}^{\infty}\left(1+\frac{1}{p^{3^i}}\right)\nonumber\\
\ & = \ \prod_{p}\left(1-\frac{1}{p^2}\right) \prod_{i=1}^{\infty}\frac{1-\frac{1}{p^{2\cdot 3^i}}}{1-\frac{1}{p^{3^i}}}\nonumber\\
\ & = \ \frac{1}{\zeta(2)} \prod_{i=1}^{\infty}\frac{\zeta(3^i)}{\zeta(2\cdot3^i)}\ \approx\ 0.71974.
\end{align}

Though the asymptotic density is usually the preferred notion of size of a progression-free set when it exists, other types of density can be computed that reveal different information about the size of a set and the spacing of its elements, or that are more sensitive in the case of very small or large sets. Another way to measure the ``size'' of a set is the \emph{uniform density}, also known as Banach density, first described in \cite{BF1}.
\begin{definition}
The upper uniform density of a set $A \subseteq \N$, if it exists, is defined to be 
\be\label{eq:upperuniformdensity}
\overline{u}(A) \ = \ \lim_{s\to \infty} \max_{n\geq 0} \sum_{a \in A, n <
a \leq n + s}\frac{1}{s},
\ee and the lower uniform density of a set $A \subseteq \N$, if it exists, is defined to be
\be\label{eq:loweruniformdensity}
\underline{u}(A) \ = \ \lim_{s\to \infty} \min_{n\geq 0} \sum_{a \in A, n <
a \leq n + s}\frac{1}{s}.
\ee
\end{definition}

The uniform density exists if the upper and lower uniform densities are the same, in which case $u(A) = \overline{u}(A) = \underline{u}(A)$. Intuitively, the uniform density measures how sparse or dense a set can be locally.  Notice that uniform density is more sensitive than asymptotic density, specifically to local densities in any interval past the initial interval. This is particularly helpful to us because our sets tend to have increasing gaps between elements.    For more information and background on the uniform density see \cite{BF2,GLS,Gr}.
For any set $A$ of natural numbers we have (see \cite{Gr})
\be\label{densityinequalities}
0 \ \leq \ \underline{u}(A) \ \leq \ \underline{d}(A) \ \leq \ \overline{d}(A) \ \leq \ \overline{u}(A) \ \leq \ 1.
\ee

Furthermore, notice that if both the uniform and the asymptotic densities exist, then they are equal.
These values can differ substantially, however.  It is shown in \cite{Mi} that for any $0\leq \alpha \leq \beta \leq \gamma \leq \delta \leq 1$ there exists a set of integers, $A$, with $\underline{u}(A)=\alpha$, $\underline{d}(A)=\beta$, $\overline{d}(A)=\gamma$ and $\overline{u}(A)=\delta$.  
\subsection{Sets Free of Geometric Progressions} In \cite{McNew} a set $S$ is constructed to have high upper asymptotic density as follows.  For any fixed $N$, let
\be\label{eq:nathancheese}
\mathbb{S}_N\ =\ \left(\frac{N}{48}, \frac{N}{45}\right] \bigcup \left(\frac{N}{40}, \frac{N}{36}\right] \bigcup \left(\frac{N}{32}, \frac{N}{27}\right] \bigcup \left(\frac{N}{24}, \frac{N}{12}\right] \bigcup \left(\frac{N}{9}, \frac{N}{8}\right] \bigcup \left(\frac{N}{4}, N\right].
\ee

Now, fix $N_1= N$, let
\be
N_i\ =\ \frac{48^2N_{i-1}^2}{N_1},
\ee
and let $S$ be the union of all such $S_{N_i}$. This set is free of geometric progressions with integral ratios and has upper asymptotic density approximately 0.815509. However, its lower asymptotic density, and therefore its lower uniform density, is 0, and it is readily seen that its upper uniform density is 1, because $S$ contains arbitrarily long stretches of included numbers.

An open problem, stated by Beiglb\"ock, Bergelson, Hindman, and Strauss \cite{BBHS}, asks whether it is possible to find a set of integers free of geometric progressions such that the number of consecutive excluded terms is bounded. (Such a set is sometimes called syndetic.)  Using a Chinese remainder theorem-type argument one find that Rankin's greedy set does not have this property. To see this, let $p_n$ denote the $n$-th prime number and consider the congruences
\bea\label{rankingaps}
a &\equiv& p_1^2 \pmod {p_1^3}\nonumber\\
a + 1 &\equiv& p_2^2 \pmod {p_2^3}\nonumber\\
& \vdots& \nonumber\\
a + n - 1 &\equiv& p_{n}^2 \pmod {p_{n}^3}.
\eea
By the Chinese remainder theorem, there exists an integer $a$ that satisfies these congruences, so that the $n$ consecutive  integers $a, \dots, a+n-1$ are all excluded from Rankin's greedy set.

The problem above is equivalent to asking whether there exists a set of integers with positive lower uniform density which avoid geometric progressions, which leads us to consider the uniform density of similar sets. This problem has also been considered recently by \cite{He}.



\subsection{Upper Uniform Density of Rankin's Set}


We know the asymptotic density of Rankin's set, $G_3^*$, as well as its lower uniform density by the argument above. We now consider the upper uniform density of Rankin's set, starting with a simple upper bound.

\begin{theorem}
An upper bound on the upper uniform density of $G_3^*$ is $7/8$.
\end{theorem}

\begin{proof}
Note that all integers that are exactly divisible by $2^2$ are excluded from Rankin's set. That is, all integers that are congruent to $4 \bmod 8$ are excluded. We know that $\overline{u}(\{x : x \not \equiv 4 \bmod 8\}) = 7/8$, and therefore we have that $\overline{u}(G_3^*) \leq 7/8$.
\end{proof}

By extending this kind of argument to primes other than 2 and powers greater than 2 which must also be excluded we are able to determine the exact upper uniform density of this set.  Enumerate the primes by $\{p_j\}_{j=1}^\infty$ and recall that for any prime $p$, if any $n$ in our set is exactly divisible by $p^k$ for some $k$ in $A_3^*$ then $n$ is excluded from Rankin's set.





\begin{theorem} The upper uniform density of $G_3^*$ equals its asymptotic density: $\ol u (G_3^*)= d(G_3^*)$. \label{thm:uudrankin}
\end{theorem}

\begin{proof} By \eqref{densityinequalities} we know that $d(G_3^*) \leq \ol u (G_3^*)$. Thus to prove our result it is sufficient to show that $\ol u (G_3^*) \leq d (G_3^*)$.

Let
\be T_i \ := \ \{k \ : \ p_j^b \ | \ k \Rightarrow p_j^{b+1} \ | \ k \; \text{holds for all } \; j\leq i \;\text{ and } b \leq i, b \notin A_3^* \} \ee
be the set of integers not exactly divisible by any of the first $i$ primes raised to a power (at most i) that is not in $A_3^*$.

Then, as a first step, notice that the proportion of integers not exactly divisible by $p_j^2$ in any interval of length $p_j^3$ is $\left(1-\frac{1}{p_j^2}+\frac{1}{p_j^3}\right)$. Generalizing this, the proportion of integers not exactly divisible by $p_j^b$ for any $b \leq i$, that is not in $A_3^*$ in any interval of length $p_j^{i+1}$ is
\be 1 - \sum_{\substack{0 \leq b\leq i\\ b \notin A_3^*}}\left( \frac{1}{p_j^b} - \frac{1}{p_j^{b+1}} \right).
\ee
Thus, by the Chinese Remainder Theorem, the proportion of integers contained in $T_{i}$ in any interval of length $\prod_{j=1}^i p_j^{i+1}$ is
\be \prod_{j=1}^i \left(1 - \sum_{\substack{0\leq b\leq i\\ b \notin A_3^*}}\left( \frac{1}{p_j^b} - \frac{1}{p_j^{b+1}}\right) \right), \label{smprop2} \ee
so \eqref{smprop2} gives the uniform density of $T_{i}$, and thus the upper uniform density as well.

Because $G_3^* \subset T_i$ for each $i$, we have $\overline{u}(G_3^*) \leq \overline{u}(T_i)$ for each $i$.  Using the expression \eqref{smprop2} for $\ol u(T_{i})$, and letting $i$ go to infinity,
\begin{align}
\ol u (G_3^*) \ \leq \ \ol \lim_{i \to \infty} u (T_i) \ & = \ \lim_{i\to\infty} \prod_{j=1}^i \left(1 - \sum_{\substack{0 \leq b\leq i\\ b \notin A_3^*}}\left( \frac{1}{p_j^b} - \frac{1}{p_j^{b+1}}\right) \right) \nonumber \\
& = \  \prod_{j=1}^\infty \left(1 - \left(1 - \frac{1}{p_j} \right) \sum_{\substack{ b \in \N \setminus A_3^*}} \frac{1}{p_j^b} \right) \nonumber \\
\ & = \ \prod_{j=1}^\infty \left(1 - \frac{1}{p_j} \right) \left[ \sum_{b=0}^\infty \frac{1}{p_j^b} - \sum_{b\in \N \setminus A_3^*} \frac{1}{p_j^b} \right] \nonumber \\
\ & = \ \prod_{j=1}^\infty \left(1 - \frac{1}{p_j} \right) \sum_{a\in A_3^*} \frac{1}{p_j^a} = d(G_3^*).
\end{align}
\end{proof}

\section{Greedy Set Avoiding Exponential Progressions} \label{sec:expo}
We can view both a 3-term arithmetic progression and a 3-term geometric progression as a sequence $x, f_n(x), f_n(f_n(x))$, where $f_n(x) = x + n$ or $f_n(x) = nx$ respectively. We can similarly construct other sequences in terms of different families of functions. We consider first exponential progressions with $f(x) = x^n$.

\subsection{Characterization and Density}
Let $E_3^* = \{1, 2, 3,\dots, 14, 15, 17, \dots, 79, 80, 82, \dots\}$ be the greedy set of integers free of exponential progressions; that is, $E_3^*$ avoids progressions of the form
$x, x^n, x^{n^2}$, 
where $x, n$ are natural numbers greater than 1.

\begin{proposition}
An integer $k = p_1^{a_1}p_2^{a_2} \cdots p_n^{a_n}$ is included in $E_3^*$ if and only if $g = \gcd (a_1,a_2,\dots,a_n)$ is included in $G_3^*$.
\end{proposition}

\begin{proof} We proceed by induction on $k$. Clearly, $1 \in E_3^*$. Assume that for all integers less than $k$, our inductive hypothesis holds, and that $k= p_1^{a_1}p_2^{a_2} \cdots p_n^{a_n}$, with $g = \gcd (a_1,a_2,\dots,a_n)$.

Suppose first that $g \notin G_3^*$. Since $g$ is excluded from $G_3^*$, it must be the last term of a geometric progression. Thus, there exists a natural number $r > 1$ such that $g/r^2, g/r, g$ is a geometric progression with $g/r^2$ and $g/r$ both in $G_3^*$. Setting $b_i = a_i/r$, it is clear that $\gcd (b_i) = g/r$, and by the inductive hypothesis, the number $k_1 = p_1^{b_1}p_2^{b_2}\cdots p_n^{b_n}$ is in $E_3^*$. Similarly, if we set $c_i = a_i/r^2$, it is clear that $\gcd (c_i) = g/r^2$, and by the inductive hypothesis, it follows again that the number $k_0 = p_1^{c_1}p_2^{c_2}\cdots p_n^{c_n}$ is in $E_3^*$. Then, since $k_0^r = k_1$ and $k_1^r = k$, it follows that $k_0, k_1, k$ is an exponential progression, so that $k$ is not in $E_3^*$.

Now suppose that $k \notin E_3^*$. Since $k$ is excluded from $E_3^*$, it must be the last term of an exponential progression; thus there exists a natural number $m > 1$ such that $\sqrt[m^2] k, \sqrt[m] k, \ k$ is an exponential progression with the first two terms in $E_3^*$. In particular, since
\[ \sqrt[m^2] k \ = \ p_1^{\frac{a_1}{m^2} }p_2^{\frac{a_2}{m^2}} \cdots p_n^{\frac{a_n}{m^2}} \in E_3^*, \]
we have by the inductive hypothesis that the number
\[ g/m^2 \ = \ \gcd \left ( \frac{a_1}{m^2}, \frac{a_2}{m^2}, \ldots, \frac{a_n}{m^2} \right ) \]
is in $G_3^*$. Similarly, $g/m \in G_3^*$. Then, since $g/m^2, g/m, g$ is a geometric progression, it follows that $g$ is not in $G_3^*$.
\end{proof}

Throughout the subsequent sections we will refer to the set of squareful numbers.

\begin{definition}
An integer $n$ is squareful if, for any prime $p$ dividing $n$, $p^2$ also divides $n$.
\end{definition}

Unlike the cases of arithmetic progressions and geometric progressions, where the greedy sets are not necessarily optimal, we find that it is not really possible to do better than $E_3^*$ while avoiding exponential progressions.  We see first that $E_3^*$ already has uniform (and asymptotic) density 1.


\begin{theorem}\label{ud} We have $u(E_3^*)= d(E_3^*) = 1$.
\end{theorem}


\begin{proof}
With Equation \eqref{densityinequalities} in mind, we prove that the uniform density of $E_3^*$ is 1 by showing that the lower uniform density is 1. Equivalently, we show that the upper uniform density of $\mathbb{N}\setminus E_3^*$ is 0. Since $\mathbb{N}\setminus E_3^*$ is a subset of the squareful numbers, it is sufficient to show that the upper uniform density of the squareful numbers is 0, which we do by considering yet another superset, namely, the set of numbers not exactly divisible by the first power of any small primes.

Let
\be R_i \ := \ \{k \ : \ p_j \ | \ k \Rightarrow p_j^2 \ | \ k \; \text{holds for all } \; j\leq i \} \ee
be the set of integers not exactly divisible by any of the first $i$ primes to the first power.
Notice that the proportion of integers not exactly divisible by $p_j$ in any interval of length $p_j^2$ is $\left(1-\frac{1}{p_j} +\frac{1}{p_j^2}\right)$.
Thus, by the Chinese Remainder Theorem, the proportion of integers contained in $R_i$ in any interval of length $\prod_{j=1}^i p_j^2$ is

\be \prod_{j=1}^i \left(1-\frac{1}{p_j}+\frac{1}{p_j^2}\right), \label{ethreeprop}   \ee
so \eqref{ethreeprop} gives the uniform density of $R_i$, and thus the upper uniform density as well.

Because $\mathbb{N}\setminus E_3^* \subset R_i$ for each $i$, we have $\overline{u}(\mathbb{N}\setminus E_3^*) \leq \overline{u}(R_i)$ for each $i$. Using the expression \eqref{ethreeprop} for $\ol u(R_{i})$, and letting $i$ go to infinity,
\begin{align}
\ol u (\mathbb{N}\setminus E_3^*) \ \leq \ \lim_{i \to \infty} u (R_i) \ & = \prod_{j=1}^\infty \left(1-\frac{1}{p_j}+\frac{1}{p_j^2}\right) \ =\ 0
\end{align}

Thus we must have that $\underline{u}(E^*_3) = 1-\ol{u}(\N \setminus E_3^*) =1$ and so both the uniform and asymptotic densities of $E_3^*$ are 1.
\end{proof}

Because $E_3^*$ has density 1, we focus now on the excluded set of integers, $\N\setminus E_3^*$, which has density 0, and ask whether it is possible to do better, creating a set which avoids exponential progressions while excluding fewer integers.  Using the exponential density, which can be used to further analyze sets with density zero, we will see that $E_3^*$ is essentially best possible.

\begin{definition} The upper exponential density of a set $A \subseteq \mathbb{N}$ is defined to be
\be\label{eq:expodensity}
\overline{e}(A) \ := \ \limsup_{n \to \infty} \frac{1}{\log(n)} \log \left( \sum_{a \in A, a \leq n} 1 \right).
\ee
The lower exponential density $\underline{e}$, and the exponential density $e$ are defined similarly in the usual way.
\end{definition}

Note that the exponential density is defined such that the $k$th-powers have density $1/k$, and that any set with positive lower asymptotic density will have exponential density 1.

\begin{proposition} \label{ed} The exponential density of the set of integers excluded from the greedy exponential-progression-free set is $e(\N\setminus E_3^*)= 1/4$.
\end{proposition}

\begin{proof}
We first consider exponential progressions, $x, x^n, x^{n^2}$, in the case when $n=2$, the smallest non-trivial case.  We will see that numbers excluded from this sort of progression form the bulk of the numbers that are excluded.

In the interval $[1, N]$, a first approximation for the number of integers that are excluded from $E_3^*$ is $N^{1/4}$. If $m \leq N^{1/4}$, then $m^4 \leq N$ and there is a progression of the form $m, m^2, m^4$. However, not every fourth power is thus excluded. Specifically, if $m$ or $m^2$ is already excluded from $E_3^*$ then $m^4$ will not be. For example, $4^4=2^8$ would not be excluded even though it is a fourth power, since $2^4$ is already excluded.

Because this situation only occurs when the initial term, $m$, of an exponential progression is already part of a smaller progression, and thus a number, we account for this sort of integer with an error term counting all the squareful numbers less than $N^{1/4}$. The count of the squareful numbers up to $M$ is  $O\left(M^{1/2}\right)$, see for example \cite{Go}, so the number of squareful numbers up to $N^{1/4}$ is $O\left(N^{1/8}\right)$.  Thus,  simply looking at the exponential progressions where $n=2$, we exclude $\sqrt[4]{N}+O(\sqrt[8]{N})$ elements from the interval $[1, N]$. 

Moreover, for each $n > 2$, we see that the number of excluded terms due to progressions $x,x^n,x^{n^2}$ is $O\left(N^{1/n^2}\right)$ which is smaller than the error term in the expression above.

Finally, we use this to compute the exponential density of $\N\setminus E_3^*$,
\begin{align} \label{eq:expdensity}
e(\N\setminus E_3^*) \ & = \ \lim_{N \to \infty} \frac{\log({\sqrt[4]{N} + O(\sqrt[8]{N})})}{\log{N}} \nonumber \\
\ & = \ \lim_{N \to \infty} \frac{\log(
\sqrt[4]{N}(1+O({N^{-1/8}})))}{\log{N}}  = \ \frac{1}{4}.
\end{align}

\end{proof}

Note that, any set that avoids exponential progressions will have to exclude on the order of $\sqrt[4]{N}$ terms to account for fourth powers, producing a set of excluded integers with exponential density at least 1/4. So we see that in this sense, $E_3^*$ is the optimal set containing no exponential progressions.

\section{Excluded Exponent Sets}
Another possible way to generalize the sets $A_3^*$ and $G_3^*$ is to iterate the method used to construct $G_3^*$ by taking taking those numbers whose prime exponents are contained in $A_3^*$.  We can construct a third set of numbers where the admissable prime exponents are the integers in $G_3^*$. By repeating this pattern, we construct a family of sets.



\subsection{Characterization and Density}
We obtain the $n$\textsuperscript{th}  set by taking all of the numbers whose primes are raised only to the powers in the $(n-1)$\textsuperscript{th} set. Let $S_n$ be the $n$th set so constructed, where $S_1=A_3^*$ and $S_2=G_3^*$.



\subsection{Density of Iterated Construction}
We consider the asymptotic densities of sets with this construction, and then we consider the lower uniform densities. First, we define a generalization of the square-free numbers and prove two results useful for our discussion.

\begin{definition} Let $x > 0$ be an integer with factorization $p_1^{a_1}p_2^{a_2}\cdots p_n^{a_n}$. We say $x$ is \emph{$k$-free} if $a_i < k$ for each $1 \leq i \leq n$.
\end{definition}

\begin{lemma} \label{lem:dkfreeto1} For each $k \geq 2$, let $Q_k$ be the set of $k$-free numbers. Then $\lim_{k \to \infty} d(Q_k) = 1$.
\end{lemma}

\begin{proof}
From, for example, Pappalardi \cite{Pap} we know that 
\be
S^k(x) \ := \ \#\{n \leq x \mid n \text{ is } k\text{-free}\} \ = \ \frac{x}{\zeta(k)}+O(x^{\frac{1}{k}}),
\ee
where $\zeta$ is the Riemann zeta function. Thus we have
\begin{equation} \lim_{k \to \infty} d(Q_k) \ = \ \lim_{k \to \infty} \frac{1}{\zeta(k)} \ = \ 1. 
\end{equation}
\end{proof}

\begin{lemma}\label{lem:bigbutnotfull}
Let $B_m$ be the set of positive integers whose prime factorizations have at least one prime raised to the $m$\textsuperscript{th} power. Then $d(B_m) > 0$ for each $m \geq 2$.
\end{lemma}

\begin{proof}
To compute the density, we rewrite $B_m$ as $Q_{m+1} \setminus Q_m$. Then, since $Q_m \subset Q_{m+1}$, we have
\be d(B_m) \ = \ d(Q_{m+1}) - d(Q_m) \ = \ \frac{1}{\zeta (m+1)} - \frac{1}{\zeta (m)} \ > \ 0,
\ee
for each $m \geq 2$, as desired.
\end{proof}

\begin{theorem} \label{uad}
The asymptotic density of $S_n$ approaches $1$ as $n$ goes to infinity, but there is no $n$ for which the density of $S_n$ equals $1$. 
\end{theorem}

\begin{proof}
By the definition of $S_n$ for $n>1$, we have
\be d(S_n) = \prod_p\left(\frac{p-1}{p}\right)\sum_{i \in S_{n{-}1}} \frac{1}{p^i}.\ee

$S_n$ contains as a subset the $k$-free numbers for some $k$. As $n \to \infty$, $k \to \infty$ as well. By Lemma \ref{lem:dkfreeto1}, we know that as $k \to \infty$, the density of the $k$-free numbers approaches $1$. Thus, we get that $d(S_n) \to \infty$ as $n \to \infty$.

However, in each set, there exists some $m$ such that no numbers whose factorizations have a prime raised to the $m$\textsuperscript{th} power are included. The set of numbers with at least one prime raised to the $m$\textsuperscript{th} power has positive density by Lemma \ref{lem:bigbutnotfull}. Thus no set in this family has density $1$.
\end{proof}


Nevertheless, the sets $S_n$ increase in size very quickly. For example, in the fourth iteration of this family the first element that is excluded is $2^{2^{2^{2}}} = 65536$. The densities of $S_n$ for the first few values of $n$ are given in Table \ref{denstable}.

\large
\begin{center}
\begin{table}[h!]
\large
\begin{center}
\begin{tabular}{|c|c|}
\hline
\ $n$\ & \ $d(S_n)$\ \\
\hline
\ 1\ & 0 \\
\hline
2 & 0.719745 \\
\hline
3 & 0.957964 \\
\hline
4 & 0.999992  \\
\hline
\end{tabular}
\begin{center}\caption{Densities of the sets $S_n$.}
\label{denstable} \end{center}
\normalsize
\end{center}
\end{table}
\end{center}
\normalsize


Despite the high densities of these sets, they all still miss arbitrarily long sequences of consecutive integers.
\begin{theorem} \label{lud}
For each $n$, $S_n$ has lower uniform density 0.  
\end{theorem}

\begin{proof}
Fix $n>1$, and consider the set $S_n$.  Using the Chinese Remainder Theorem as in \ref{rankingaps} we can construct an arbitrarily long sequence of consecutive numbers all of which are excluded from $S_n$.

Let $m$ be a number excluded from $S_{n-1}$. Then any number with a prime raised to the $m$\textsuperscript{th} power in its prime factorization is excluded from $S_n$. We construct a list of $l$ numbers each of which is exactly divisible by some prime raised to the $m$\textsuperscript{th} power. Take the first $l$ primes, $p_1, \dots, p_{l}$ and consider the system of congruences

\begin{align}
a \ & \equiv \ p_1^{m} \pmod {p_{1}^{m+1}} \nonumber \\
a+1 \ & \equiv \ p_2^{m} \pmod {p_{2}^{m+1}} \nonumber \\
& \; \; \vdots \nonumber \\
a+l-1 \ & \equiv \ p_{l}^{m} \pmod {p_{l}^{m+1}}.
\end{align}

By the Chinese Remainder Theorem, there exists an $a$ that solves this system of congruences, and so the integers $a,a+1,\ldots a+l-1$ are all excluded from $S_n$.
\end{proof}

Note that an argument analogous to that of the proof of Theorem \ref{thm:uudrankin} would show that the upper uniform density of $S_n$ is equal to its asymptotic density.


\section{Conclusion and Future Work}

In addition to calculating the upper uniform density of Rankin's set, we have characterized the greedy set of integers avoiding 3-term exponential progressions, and analyzed it using the asymptotic, uniform and exponential densities.  We have also generalized the construction of the set $G_3^*$ and analyzed the densities of the resulting sets.


We end with some additional topics we hope to pursue in later work.


\begin{que}
What other functions $f_n(x)$ could we use to define interesting progression-free sets?  How does the resulting progression-free set depend on the function?
\end{que}

\begin{que}
Can the sets $S_n$ be characterized as being free of some form of progression or pattern?
\end{que}

\begin{que}
What other notions of density reveal meaningful information about the size of a progression-free set? The multiplicative density, defined by Davenport and Erd\H{o}s \cite{DE}, might be particularly interesting to consider. How does the use of different measures of density affect the structure of an optimal progression-free-set?
\end{que}

\begin{que}
One might consider a family of sets where the set after $E_3^*$ extends from $E_3^*$ analogously to how $E_3^*$ extends from $G_3^*$, that is, an integer $k = p_1^{a_1}p_2^{a_2} \cdots p_n^{a_n}$ is included in the $n$\textsuperscript{th} set if and only if $g=\gcd (a_1,a_2,\dots,a_n)$ is included in the $(n-1)$\textsuperscript{th} set. Can the sets after the first three in this family be characterized as avoiding some meaningful kind of progression?
\end{que}

\begin{que}
What about exponential-progression-free sets over Gaussian integers? Or other number fields? In particular, what can be said about the densities of the sets of ideals which avoid exponential progressions?
\end{que}



$ $


\begin{thebibliography}{BHMMPTW} 

\bibitem[BBHS]{BBHS}
V. Bergelson, M.Beiglb\"{o}ck, N. Hindman and D. Strauss, \emph{Multiplicative structures in additively large sets}, J. Comb. Theory (Series A) \textbf{113} (2006), 1219--1242, \bburl{http://www.sciencedirect.com/science/article/pii/S0097316505002141}.\\

\bibitem [BF1]{BF1}
T. C. Brown, and A. R. Freedman, \emph{Arithmetic progressions in lacunary sets}, Rocky Mountain J. Math
\textbf{17} (1987), 587--596.\\

\bibitem [BF2]{BF2}
T. C. Brown, and A. R. Freedman, \emph{The Uniform Density of sets of integers and Fermat's Last Theorem}, C. R. Math. Rep. Acad. Sci. Canada \textbf{12} (1990), 1--6, \bburl{http://people.math.sfu.ca/~vjungic/tbrown/tom-37.pdf}.\\

\bibitem [BHMMPTW]{BHMMPTW}
A. Best, K. Huan, N. McNew, S. J. Miller, J. Powell, K. Tor, and M. Weinstein, \emph{Geometric-Progression-Free Sets Over Quadratic Number Fields},  preprint (2014), \bburl{http://arxiv.org/pdf/1412.0999v1.pdf}.\\

\bibitem [DE]{DE}  H. Davenport and P. Erd\H{o}s, \emph{On sequences of positive integers}, J. Indian Math. Soc.\textbf{ 15} (1951), 19--24.\\

\bibitem [GLS]{GLS}
Z. G\'{a}likov\'{a}, B. L\'{a}szl\'{o}, and T. Sal\'{a}t, \emph{Remarks on Uniform Density of Sets of Integers}, Acta Acad. Paed. Agriensis, Sectio Mathematicae \textbf{29} (2002), 3--13, \bburl{http://www.kurims.kyoto-u.ac.jp/EMIS/journals/AMI/2002/acta2002-galikova-laszlo-salat.pdf}.\\

\bibitem [Go]{Go} S. W. Golomb,  \emph{Powerful Numbers}, Amer. Math. Monthly \textbf{77} (1970), 848--855.\\

\bibitem [Gr]{Gr}
G. Grekos, \emph{On various definitions of density (survey)}, Tatra Mt. Math. Publ. \textbf{31} (2005), 17--27, \bburl{http://www.mis.sav.sk/journals/uploads/0131150501GREK02.ps}.\\

\bibitem [GTT]{GTT}
G. Grekos, V. Toma, and J. Tomanov\'{a}, \emph{A note on uniform or Banach density}, Annales Math\'{e}matiques Blaise Pascal. \textbf{17} (2010), 153--163, \bburl{http://ambp.cedram.org/cedram-bin/article/AMBP\textunderscore 2010\textunderscore \textunderscore 17\textunderscore 1 \textunderscore 153\textunderscore 0.pdf}.\\

\bibitem [He]{He}
X. He, \emph{Geometric Progression-Free Sequences with Small Gaps}, J. Number Theory \textbf{151} (2015) 197--210.\\

\bibitem[Mi]{Mi}
L. Mi\v{s}\'{i}k, \emph{Sets of positive integers with prescribed values of densities}, Math. Slovaca \textbf{52} (2002), 289--296.\\

\bibitem[McN]{McNew}
N. McNew,  \emph{On sets of integers which contain no three terms in geometric progression}, Math. Comp., electronically published on May 14, 2015, DOI: \bburl{http://dx.doi.org/10.1090/mcom/2979} (to appear in print).\\

\bibitem[NO1]{Nato'Bry1}
M. B. Nathanson and K. O'Bryant, \emph{A problem of Rankin on sets without geometric progressions}, preprint (2014), \bburl{http://arxiv.org/pdf/1408.2880.pdf}\\

\bibitem[NO2]{NatO'Bry2}
M. B. Nathanson and K. O'Bryant, \emph{On sequences without geometric progressions}, Integers \textbf{13} (2013), Paper No. A73.\\

\bibitem[Ran]{Rankin}
R. A. Rankin, \emph{Sets of integers containing not more than a given number of terms in arithmetical progression}, Proc. Roy. Soc. Edinburgh Sect. A \textbf{65} (1960/61), 332--344 (1960/61). \\

\bibitem[Rid]{Riddell}
J. Riddell, \emph{Sets of integers containing no $n$ terms in geometric progression}, Glasgow Math. J. \textbf{10} (1969), 137--146.\\

\bibitem[Pap]{Pap}
F. Pappalardi, \emph{A Survey on $k-$Freeness}, Number theory, Ramanujan Math. Soc. Lect. Notes Ser., vol. 1, Ramanujan Math. Soc., Mysore, 2005, pp. 71--88.\\


\end{thebibliography}
\end{document}